\pdfoutput=1

\documentclass[leqno,11pt]{amsart}

\usepackage{amsmath}

\usepackage{amssymb}

\newtheorem{thm}{Theorem}[section]

\newtheorem{thmx}{Theorem} 

\newtheorem{prop}[thm]{Proposition}
\newtheorem{lemma}[thm]{Lemma}

\theoremstyle{definition}
\newtheorem{defn}[thm]{Definition}
\newtheorem{ex}[thm]{Example}

\theoremstyle{remark}
\newtheorem{remark}[thm]{Remark}

\numberwithin{equation}{section}

\def\C{\mathbb{C}}
\def\R{\mathbb{R}}

\def\SO{\mathrm{SO}}
\def\GL{\mathrm{GL}}

\def\Q{\mathcal{Q}}

\def\P{\mathcal{P}}

\def\span{\mathrm{span\,}}

\begin{document}

\title[]{Holomorphic differentials and Laguerre deformation of surfaces}



\author{Emilio Musso}
\address{(E. Musso) Dipartimento di Scienze Matematiche, Politecnico di Torino,
Corso Duca degli Abruzzi 24, I-10129 Torino, Italy}
\email{emilio.musso@polito.it}

\author{Lorenzo Nicolodi}
\address{(L. Nicolodi) Di\-par\-ti\-men\-to di Ma\-te\-ma\-ti\-ca e Informatica,
Uni\-ver\-si\-t\`a degli Studi di Parma, Parco Area delle Scienze 53/A,
Campus Universitario, I-43100 Parma, Italy}
\email{lorenzo.nicolodi@unipr.it}

\thanks{Authors partially supported by MIUR (Italy) under the PRIN project
\textit{Variet\`a reali e complesse: geometria, topologia e analisi armonica};
and by the GNSAGA of INDAM}





\begin{abstract}
A Laguerre geometric local characterization is given
of $L$-minimal surfaces and Laguerre deformations ($T$-transforms) of $L$-mi\-ni\-mal isothermic surfaces in terms of the holomorphicity of a quartic and a quadratic differential.
This is used to prove that,
via their $L$-Gauss maps, the $T$-transforms
of $L$-minimal isothermic surfaces
have constant mean curvature $H=r$
in some translate of hyperbolic 3-space $\mathbb H^3(-r^2)\subset \R^4_1$,
de Sitter 3-space $\mathbb S^3_1(r^2)\subset \R^4_1$,
or have
 mean curvature $H=0$ in some translate of a time-oriented
lightcone in $\R^4_1$.
As an application, we show that various instances of the Lawson isometric correspondence
can be viewed as special cases of the $T$-transformation of $L$-isothermic surfaces with holomorphic quartic differential.

\end{abstract}

\maketitle

\section{Introduction}\label{s:intro}

Many features of constant mean curvature (CMC) surfaces in 3-di\-men\-sio\-nal space forms,
viewed as isothermic surfaces in M\"obius space $S^3$, can be interpreted in terms of the transformation
theory of isothermic surfaces.\footnote{We recall that a surface is
isothermic if it admits conformal curvature line coordinates away from umbilic points.}
For instance, the Lawson correspondence between CMC surfaces in space forms
can be viewed as a special case of the classical $T$-transformation of isothermic surfaces.
More specifically, it is shown that CMC surfaces in space forms arise in associated 1-parameter families
as $T$-transforms of minimal surfaces in space forms \cite{Bianchi1905-12,Ca1903,Ca1915,CM,HJlibro}.
In addition to being isothermic,
minimal surfaces in space forms
are Willmore,
that is, are critical points of the Willmore energy
$\int(H^2- K) dA$, where $H$ and $K$ are the mean and
(extrinsic) Gauss curvatures, and $dA$ is the induced area element
of the surface \cite{Blaschke,Br-duality,Thomsen}.
By a classical result of Thomsen \cite{Thomsen}, a Willmore surface
without umbilics
is isothermic if and only if
it is locally M\"obius equivalent to a minimal surface in
some space form.
According to  \cite{Bo-Pe2009, Bo2012}, K. Voss obtained
a uniform M\"obius geometric characterization of
Willmore  surfaces and CMC surfaces in space forms
using the quartic differential $\Q$ introduced by
Bryant \cite{Br-duality} for
Willmore surfaces.
Voss observed that $\Q$,
which indeed may be defined for any conformal immersion of a Riemann surface $M$
into $S^3$,
is holomorphic
if and only if, locally and away from umbilics and isolated points,
the immersion is Willmore or has constant mean curvature in some
space form embedded in $S^3$.

\vskip0.2cm

The purpose of this paper is to discuss the Laguerre geometric counterpart
of the M\"obius
situation described above.
One should consider that,
already in the fundamental work
of Blaschke and Thomsen \cite{Blaschke},
M\"obius and Laguerre surface geometries were
developed in parallel, as subgeometries of Lie sphere geometry.
Another motivation is
that several classical topics in Laguerre geometry,
including Laguerre minimal surfaces and
Laguerre isothermic surfaces,
have recently received much
attention in the theory of integrable systems,
in discrete differential geometry, and in the applications
to geometric computing and architectural geometry
\cite{Bobenko2006, Bobenko2007, Bobenko2010, li-wang-mm, MN-REND-RM,
MN-TAMS, MN-BOLL, MN-IJM, Pott1998, Pott2009, Pott2012}.

\vskip0.2cm

Let us begin by recalling some facts about the Laguerre geometry of surfaces in $\R^3$
to better illustrate our
results.
The group of Laguerre geometry consists of those
transformations
that map oriented planes in $\R^3$ to oriented planes,
oriented spheres (including points) to oriented spheres (including points), and preserve
oriented contact.
As such, the Laguerre group is a subgroup
of the group of Lie sphere transformations
and is isomorphic to the
10-dimensional restricted Poincar\'e group
\cite{Blaschke, Ce}.
Any smooth immersion of an oriented surface into $\R^3$ has a
Legendre (contact) lift to the
space of contact elements
$\Lambda =\R^3\times S^2$.
The Laguerre group acts on the Legendre lifts rather than on the
immersions themselves, since it does not act by point-transformations.
The {\it Laguerre space} is $\Lambda$ as homogeneous space of the Laguerre group.
The principal aim of Laguerre
geometry is to study the properties of an immersion which are invariant
under the action of the Laguerre group on Legendre surfaces.
Locally and up to Laguerre transformation, any
Legendre immersion arises as a Legendre lift.

A smooth immersed surface in $\R^3$ with no parabolic points is
{\it Laguerre minimal} ($L$-{\it minimal}) if it is an extremal of the
Weingarten functional $$\int(H^2/K- 1)dA,$$ where $H$ and $K$ are the mean and
Gauss curvatures of the immersion, and $dA$ is the induced area element
of the surface \cite{Blaschke1, Blaschke, MN-TAMS, Pa1999}. The functional
and so its critical points are preserved by the Laguerre group.
A surface in $\R^3$
is {\it $L$-isothermic} if, away from parabolic and umbilic points, it
admits a conformal curvature line parametrization
with respect to the third fundamental form \cite{Blaschke, MN-BOLL}.
$L$-isothermic surfaces are invariant under the Laguerre group.
See \cite{MN-BOLL, MN-IJM} for a recent study on $L$-isotermic surfaces
and their transformations,
including the analogues of the $T$-transformation
and of the Darboux transformation in M\"obius geometry.

In our discussion we will adopt the cyclographic model
of Laguerre geometry \cite{Blaschke, Ce}. Accordingly,
the space
of oriented spheres  and points in $\R^3$
is naturally identified with Minkowski 4-space $\R^{4}_1$
and the Laguerre space $\Lambda$
is described as the space of isotropic (null) lines in $\R^4_1$.
For
a Legendre immersion $F=(f,n) : M\to \Lambda$,
the {\it Laguerre Gauss} ({\it $L$-Gauss}) {\it map} of $F$,
$\sigma_F : M \to \R^4_1$,
assigns to each $p\in M$ the point of $\R^4_1$
representing
the {\it middle sphere}, that is, the
oriented sphere of $\R^3$
of radius $H/K$ which is in oriented contact with the tangent plane
of $f$ at $f(p)$.
Away from umbilics and parabolic points, $\sigma_F$ is a
spacelike immersion with isotropic mean curvature vector;
moreover, $\sigma_F$ has zero mean curvature vector
in $\R^4_1$ if and only if $F$ is $L$-minimal \cite{Blaschke, MN-TAMS}.

In \cite{MN-TAMS}, 
for a Legendre immersion $F=(f,n) : M \to \Lambda$ which is $L$-minimal,
we introduced a Laguerre invariant holomorphic quartic differential $\Q_F$ on $M$
viewed as a Riemann surface with the conformal structure induced by
 $dn\cdot dn$.
This quartic differential $\Q_F$
may be naturally defined for arbitrary
nondegenerate\footnote{cf. Section \ref{ss:iso} for the right definition.}
Legendre surfaces
together with an invariant quadratic differential $\P_F$.
The quadratic differential $\P_F$ is holomorphic when
$\Q_F$ is holomorphic and vanishes if $F$ is $L$-minimal (cf. Section \ref{s:thm:A}).

\vskip0.2cm

The first main result of this paper provides a
Laguerre geometric characterization of
Legendre surfaces
with holomorphic quartic differential.


\begin{thmx}\label{thm:A}
The quartic differential $\Q_F$ of
a nondegenerate Legendre immersion $F : M \to  \Lambda$ is holomorphic
if and only if the immersion $F$ is $L$-minimal,
in which case $\P_F$ is zero,
or is
locally the $T$-transform of an $L$-minimal isothermic
surface.
\end{thmx}

The $T$-transforms in the statement of Theorem \ref{thm:A} can be seen as
second order Laguerre deformations
\cite{MN-BOLL, MN-TOH} in the sense of Cartan's
general deformation theory \cite{Ca1, Gr, J}.
Using Theorem \ref{thm:A} we then
characterize $L$-minimal isothermic surfaces and their $T$-transforms
in terms of the differential geometry of their $L$-Gauss maps.
We shall prove the following.

\begin{thmx}\label{thm:B}
Let $F : M \to  \Lambda$ be a nondegenerate Legendre immersion.
Then:
\begin{enumerate}

\item
$F$ is $L$-minimal and $L$-isothermic
if and only if
its $L$-Gauss map $\sigma_F : M \to \R^4_1$
has zero mean curvature in some
spacelike, timelike, or (degenerate) isotropic hyperplane of $\R^4_1$.

\item $F$
has holomorphic $\Q_F$ and non-zero $\P_F$
if and only if
its $L$-Gauss map $\sigma_F : M \to \R^4_1$
has constant mean curvature $H=r$
in some translate of hyperbolic 3-space $\mathbb H^3(-r^2)\subset \R^4_1$,
de Sitter 3-space $\mathbb S^3_1(r^2)\subset \R^4_1$,
or has zero mean curvature in some translate of a time-oriented lightcone
$\mathcal L^3_\pm\subset \R^4_1$.

\end{enumerate}

\noindent In addition, if the $L$-Gauss map of $F$
takes values in a spacelike (respectively, timelike, isotropic) hyperplane, then the
$L$-Gauss maps of the $T$-transforms of $F$ take values
in a translate of a hyperbolic 3-space (respectively, de Sitter 3-space,
time-oriented lightcone).

\end{thmx}

As an application of the two theorems above, we show that
the Lawson correspondence \cite{Lawson} between
certain isometric CMC surfaces in different hyperbolic 3-spaces
and, in particular,
the Umehara--Yamada isometric perturbation \cite{UY-Crelle}
of minimal surfaces of $\R^3$ into CMC surfaces in hyperbolic 3-space,
can be viewed as a special case of the $T$-transformation of $L$-isothermic surfaces
with holomorphic quartic differential.
(For a M\"obius geometric interpretation of the Umehara--Yamada
perturbation see \cite{HMN, MN-HOUSTON}).
This interpretation also applies
to the generalizations of Lawson's correspondence in the Lorentzian
\cite{Pa1990} and the (degenerate) isotropic situations,
%
%
namely to the perturbation of maximal surfaces in Minkowski
3-space into CMC spacelike surfaces in de Sitter 3-space
\cite{AiAk, AGM, Kob, Lee2005}, and that of zero mean curvature
spacelike surfaces in a (degenerate) isotropic 3-space into zero mean curvature spacelike
surfaces in a time-oriented lightcone of $\R^4_1$.

\vskip0.2cm
The paper is organized as follows. Section \ref{s:pre} recalls some background material about Laguerre
geometry, develops the method of moving frames for Legendre surfaces,
and briefly discusses $L$-isothermic surfaces and their $T$-transforms.
Section \ref{s:thm:A} proves Theorem \ref{thm:A}.
It constructs a quartic differential $\Q_F$ and a quadratic differential $\P_F$ on $M$
from a nondegenerate Legendre immersion $F : M \to  \Lambda$,
and proves that $\Q_F$ is holomorphic if and only if $F$ is $L$-minimal or
is $L$-isothermic of a special type (cf. Proposition \ref{thm:hQ-iff-W-or-special}
and Section \ref{ss:s-L-iso}).
Section \ref{s:thm:B} proves Theorem \ref{thm:B}.
In particular, $L$-minimal isothermic surfaces
and their $T$-transforms
are characterized in terms of the differential geometry of their $L$-Gauss maps
(cf. Propositions \ref{prop:L-min-iso} and \ref{p:main}).
Section \ref{s:lawson} discusses the Laguerre deformation of surfaces
with holomorphic $\Q$ in relation with various instances of the Lawson
isometric correspondence for CMC spacelike surfaces.

\section{Preliminaries and definitions}\label{s:pre}

\subsection{The Laguerre space}
Let $\R^{4}_1$ denote Minkowski 4-space with its
structure of affine vector space and a translation invariant Lorentzian scalar
product $\langle~,~\rangle$ which
takes the form
 $$
\langle v,w\rangle=-(v^1w^4 + v^4w^1) + v^2w^2 +v^3w^3
=g_{ij}v^iw^j, 
$$
with respect to the standard basis $\epsilon_1,\dots,\epsilon_4$.
We use
the
summation convention over repeated indices.
A vector $v\in \R^4_1$ is \textit{spacelike} if $\langle v,v\rangle>0$,
\textit{timelike} if $\langle v,v\rangle<0$, \textit{lightlike} (or \textit{null} or \textit{isotropic})
if $\langle v,v\rangle=0$ and $v\neq 0$.
We fix a space orientation
by requiring that the standard
basis is positive, and fix a time-orientation by
saying that a timelike or lightlike vector $v$ is positive if
$\langle v,\epsilon_1+\epsilon_4\rangle<0$.
The corresponding positive lightcone is given by
\begin{equation}\label{time-cone}
 \mathcal L^3_+ =\left\{v\in \R^{4}_1 : \langle v,v\rangle= 0,\
\langle v,\epsilon_1+\epsilon_4\rangle<0\right\}.
\end{equation}

The {\it Laguerre group} $L$ is the group of isometries of $\R^4_1$ which preserve the given
space and time orientations. It is isomorphic to the semidirect product
$\R^4 \rtimes G$, where $G$ consists  of elements
$a=(a^i_j)\in \GL(4,\R)$ such that
\begin{equation}\label{def-L}
 \det a =1,\quad
g_{hk}a^h_ia^k_j=g_{ij},\quad
 a\epsilon_1, a\epsilon_4 \in \mathcal L^3_+.
\end{equation}
$L$ is isomorphic to the (restricted)
{\it Poincar\'e group}  $\R^4 \rtimes \SO_o(3,1)$.

By a {\it Laguerre frame} $(x;a_1,\dots,a_4)$ is meant
a position vector $x$ of $\R^4_1$
and an oriented basis $(a_1,\dots,a_4)$ of $\R^4_1$, such that
\begin{equation}\label{lframe1}
 \langle a_i,a_j\rangle = g_{ij},
 \quad
  a_1, a_4 \in \mathcal L^3_+.
\end{equation}
$L$ acts simply transitively
on Laguerre frames and the
manifold of all such frames may be identified, up to the choice of a
reference frame,
with $L$. For any $(x,a)\in L$, we regard $x$ and $a_i=a\epsilon_i$ as
$\R^4$-valued functions.
There are unique 1-forms $\omega^i_0$ and
$\omega^i_j$ such that
\begin{eqnarray}
 dx &=&\omega^i_0a_i,\label{lframe2}\\
  da_i&=&\omega^j_ia_j.\label{lframe2a}
   \end{eqnarray}

Exterior differentiation of \eqref{lframe1}, \eqref{lframe2} and \eqref{lframe2a} yields
the structure equations:
\begin{eqnarray}
 0 &=&\omega^k_i g_{kj}+\omega^k_j g_{ki}, \label{str-eq1}\\
  d\omega^i_0 &=& - \omega^i_j \wedge \omega^j_0,\label{str-eq2}\\
d\omega^i_j &=& -  \omega^i_k \wedge \omega^k_j.\label{str-eq3}
\end{eqnarray}

\vskip0.2cm
\noindent Any time-oriented isotropic line in
$\R^4_1$ may be realized as $x + tv$
where $x\in\R^4_1$,  $v\in \mathcal L^3_+$ and $t$ ranges over $\R$,
and will be denoted
by $[x,v]$. The set of all isotropic lines
$$
\Lambda =\left\{ [x,v] : x\in \R^4_1, v\in \mathcal L^3_+\right\}
$$
is called the {\it Laguerre space}. The group $L$ acts transitively
on $\Lambda$ by
$$
L\times \Lambda \to \Lambda,
\quad \left((x,a),[y,v]\right)\mapsto (x,a)\cdot [y,v] =[x + ay, av].
$$
If we choose $[0,\epsilon_1]\in \Lambda$ as an origin, and let $L_0$
be the isotropy subgroup of $L$ at $[0,\epsilon_1]$, the smooth map
$$
\pi_L : L \to\Lambda,\quad
 \pi_L(A) = A\cdot [0,\epsilon_1]=[x,a_1]
$$
is the projection map of a principal $L_0$-bundle over
$\Lambda\cong L/L_0$.
The elements of $L_0$ are matrices of the form
$$
X(d;b;x)=\left(\begin{pmatrix}
d_1\\
0\\
0\\
0\\
\end{pmatrix},
\begin{pmatrix}
 d_2&\tilde x^1&\tilde x^2&{\frac{d_2}2}{^txx} \\
 0&b^1_1&b^1_2&x^1 \\
 0&b^2_1&b^2_2&x^2 \\
 0&0&0&{\frac{1}{d_2}}
\end{pmatrix}\right),
$$
where $b=(b_j^i)\in \SO(2)$, $d=(d_1,d_2)$, $d_2 > 0$,
$x={^t(}x^1,x^2)\in\R^2$, $(\tilde x^1,\tilde x^2)=d_2{^txb}$.

\subsection{The cyclographic model of Laguerre geometry
(\cite{Blaschke, Ce, Pott1998})}\label{ss:cyclo}

The fundamental objects of Laguerre geometry in Euclidean space are
{\it oriented planes} and $L$-{\it spheres} (or {\it cycles}).
By an $L$-sphere is
meant an oriented sphere or a point (a sphere of radius zero).
The orientation is determined by specifying a unit normal vector
for planes and a signed radius in the case of a sphere.

In the {\it cyclographic model} of Laguerre geometry,
an $L$-sphere $\sigma(p,r)$, with center
$p={^t(}p^1,p^2,p^3)$ and signed radius $r\in \R$,
%
%
is represented as the point of $\R^4_1$ given by
$$
x(p,r)=
{^t\Big(}{\frac{r+p^1}{\sqrt2}},p^2,p^3,{\frac{r-p^1}{\sqrt2}}\Big).
$$
An oriented plane $\pi(n,p)$ through $p$
orthogonal to $n=(n^1,n^2,n^3)\in S^2\subset \R^3$
is identified with the
{\it isotropic} hyperplane through $x(p)\in\R^4_1$
with isotropic normal vector
$$
v(n,p)=  {^t\Big(}{\frac{1+n^1}2},{\frac{n^2}{\sqrt2}},
{\frac{n^3}{\sqrt2}},{\frac{1-n^1}2}\Big).
$$

The oriented contact of $L$-spheres and oriented planes corresponds
in $\R^4_1$ to the incidence of
points and iso\-tro\-pic hyperplanes.
Two oriented $L$-spheres represented by points $x$ and $y$ in $\R^4_1$ are in
oriented contact if and only if
$$
\langle x-y,x-y\rangle=0.
$$
In this case, $x-y$ is the normal vector of the isotropic hyperplane in
$\R^4_1$ corresponding to the common tangent plane of the $L$-spheres
represented by $x$ and $y$.

This implies
that to any time-oriented isotropic line $\ell$ corresponds
a pencil of oriented spheres
which are in oriented contact at $p(\ell)\in \R^3$ with a fixed plane
$\pi$, where $p(\ell)$ represents the
unique $x$ on $\ell$ such that $\langle x,\epsilon_1+\epsilon_4\rangle=0$.
In other words,
$\Lambda$ can be identified
with the space $\R^3\times S^2$ of oriented contact elements
of $\R^3$ by the correspondence
\begin{equation}\label{1.7}
 (p,n)\in \R^3\times S^2\mapsto[x(p),v(n,p)]\in \Lambda.
  \end{equation}
By the structure equations of $L$, we see that the 1-form
$-\langle dx,a_1\rangle$ defines an
$L$-invariant contact distribution on the Laguerre space $\Lambda$.
By \eqref{1.7}, such a contact structure
coincides with the
standard contact structure on $\R^3\times S^2$.
In this way $L$ can be seen as a 10-dimensional group of contact
transformations
acting on $\R^3\times S^2$.

The points on a spacelike line $\ell$ in $\R^4_1$ represent
$L$-spheres in Euclidean space which envelope a circular cone.
For two points $x$ and $y$ on $\ell$, we have
\[
 \langle x-y,x-y\rangle = d^2,
  \]
where $d$ is the {\it tangential distance}, that is,
the Euclidean distance between the points where any common
oriented tangent plane touches the $L$-spheres corresponding to $x$ and $y$.
The tangential distance
is zero precisely when the two $L$-spheres are in oriented contact.

The points on a timelike line $\ell$ in $\R^4_1$ represent
$L$-spheres without any common oriented tangent plane.
For two points $x$ and $y$ on $\ell$,
$$
 \langle x-y,x-y\rangle = -d^2,
  $$
where $d$ is the {\it parallel distance},
that is, the
distance between the equally oriented parallel tangent planes
to the two corresponding $L$-spheres.

The {\it spherical system of $L$-spheres}
determined by a point $z\in \R^4_1$ and a constant $c\in \R$
is the set of all $L$-spheres represented by points $x\in \R^4_1$
satisfying the equation of the {\it pseudo-hypersphere}
\begin{equation}\label{J-Ueberkugel}
  \langle x- z,x- z\rangle = c.
   \end{equation}
The system consists of all $L$-spheres which have constant
tangential or parallel distance $\sqrt{|c|}$ from the fixed $L$-sphere
represented by ${z}$.
For $c=0$, the system consists of all $L$-spheres which are in oriented
contact with the fixed $L$-sphere represented by ${z}$.
The spherical system (pseudo-hypersphere) defined by \eqref{J-Ueberkugel} is {\it isotropic}
if $c=0$, {\it timelike} if $c>0$, and {\it spacelike} if $c<0$.
%
%

The {\it planar system of $L$-spheres}
determined by a point $z\in \R^4_1$ and a vector $v\in \R^4_1$
is the set of all $L$-spheres represented by points $x\in \R^4_1$
satisfying the equation of the hyperplane
\begin{equation}\label{Ueberebene}
  \langle x- z,v\rangle = 0.
   \end{equation}
The planar system (hyperplane)
is called {\it isotropic} (respectively, {\it timelike},
{\it spacelike})
if the vector $v$ is isotropic (respectively, spacelike, timelike).

\subsection{Laguerre surface geometry: the middle frame}\label{ss:L-surf-ge}

An immersed surface $f : M \to\R^3$, oriented by a
unit normal field $n : M\to S^2$, induces
a lift  $F = (f,n)$ to $\Lambda$ which is a Legendre (contact) immersion with
respect to the canonical contact structure of $\Lambda$.
More generally, a {\it Legendre surface} is an immersed surface
$F =(f,n) : M\to\Lambda$
such that $df\cdot n = 0$. The additional condition
$dn\cdot dn>0$ will be assumed throughout. Moreover,
we say that $F$ is {\it nondegenerate} if  the quadratic forms
$df\cdot dn$ and $dn\cdot dn$ are everywhere linearly independent on $M$.
We recall that two Legendrian immersions $(M,F)$ and $(M',F')$ are
said to
be  $L$-{\it equivalent} if there exists a diffeomorphism $\phi : M\to M'$
and $A\in L$
such that  $F'\circ \phi=AF$.
Locally and up to $L$-equivalence, any Legendre surface arises as a
Legendre lift.
In particular, two immersed surfaces in $\R^3$ are $L$-equivalent if their
Legendre  lifts are $L$-equivalent.

\begin{defn}
A (local) {\it Laguerre frame field} along
a Legendre immersion $F : M\to\Lambda$ is a smooth map
$A=(a_0,a) : U \to L$
defined on an open subset $U\subset M$, such that
$\pi_L(a_0,a) = [a_0,a_1]=  F$.
\end{defn}

For any Laguerre frame field $A :  U\to L$ we let
\[
\alpha =
\left((\alpha^i_0),(\alpha^i_j)\right) = \left((A^{\ast}\omega^i_0),
(A^{\ast}\omega^i_j)\right).
\]
We then have
$$
 \alpha^4_0=0,\quad \alpha^2_1\wedge{\alpha^3_1}\neq 0.
  $$
Any other Laguerre frame field $\hat A$ on $U$ is given by
$\hat A = AX(d;b;x)$,
where $X=X(d;b;x) :  U\to L_0$ is a smooth map,
and $\hat\alpha$ and $\alpha$ are related by
$
\hat\alpha=X^{-1}\alpha X+X^{-1}dX.
$

A Laguerre frame field $A$ along $F$ is called a {\it middle frame field} if
there exist smooth functions $p_1$, $p_2$, $p_3$, $q_1$, $q_2 : U \to \R$
such that $\alpha = ((\alpha^i_0),(\alpha^i_j))$ takes the form
\begin{equation}\label{conn-form}
\left(\!\left(\!\begin{smallmatrix}
0\\
\alpha^2_0\\
\alpha^3_0\\
0\\
\end{smallmatrix}\!\right)\!,\!
\left(\!\begin{smallmatrix}
 2q_2\alpha^2_0-2q_1\alpha^3_0&p_1\alpha^2_0+p_2\alpha^3_0&
p_2\alpha^2_0+p_3\alpha^3_0&0 \\
 \alpha^2_0&0&-q_1\alpha^2_0-q_2\alpha^3_0&p_1\alpha^2_0+p_2\alpha^3_0 \\
 -\alpha^3_0&q_1\alpha^2_0+q_2\alpha^3_0&0&p_2\alpha^2_0+p_3\alpha^3_0\\
 0&\alpha^2_0&-\alpha^3_0&-2q_2\alpha^2_0+2q_1\alpha^3_0
\end{smallmatrix}\!\right)\!\right)\!,
\end{equation}
with $\alpha^2_0\wedge \alpha^3_0 > 0$; $(\alpha^2_0,\alpha^3_0)$ is
called the \textit{middle coframe} of $F$. The existence of a middle frame
field along $F$ was proved in \cite{MN-TAMS},
under the nondegeneracy assumption.
The smooth functions $q_1$,
$q_2$, $p_1$, $p_2$, $p_3$ form a complete system of Laguerre
invariants for $F$ and satisfy the following structure equations:
\begin{eqnarray}
d\alpha^2_0=q_1\alpha^2_0\wedge\alpha^3_0,&{} &
d\alpha^3_0=q_2\alpha^2_0\wedge\alpha^3_0,\label{se0} \\
dq_1\wedge\alpha^2_0+dq_2\wedge\alpha^3_0&=&
(p_3-p_1-{q_1}^2-{q_2}^2)\alpha^2_0\wedge\alpha^3_0,\label{se1}\\
dq_1\wedge\alpha^3_0-dq_2\wedge\alpha^2_0&=& - p_2\alpha^2_0\wedge\alpha^3_0,\label{se2}\\
dp_1\wedge\alpha^2_0+dp_2\wedge\alpha^3_0&
= &(-3q_1p_1-4q_2p_2+q_1p_3)\alpha^2_0\wedge\alpha^3_0,\label{se3}\\
dp_2\wedge\alpha^2_0+dp_3\wedge\alpha^3_0&=&
(-3q_2p_3-4q_1p_2+q_2p_1)\alpha^2_0\wedge\alpha^3_0.\label{se4}
\end{eqnarray}

The vector $a_0$ is independent of the middle frame field $A = (a_0, a)$ and is therefore globally
defined on $M$.

\begin{defn}
The {\it Laguerre Gauss map} ({\it $L$-Gauss map}) of the nondegenerate Legendre immersion
$F: M \to \Lambda$ is the smooth map
\[
 \sigma_F : M \to \R^4_1
  \]
defined locally by $\sigma_F := a_0 : U \to \R^4_1$, where
$A = (a_0, a): U \to L$ is a middle frame field along
$F$.
\end{defn}

\begin{remark}
If ${A}=(a_0, a)$ is a middle frame field along $F$ and $U$ is connected, the only
other middle frame field
on $U$ is given by
\[\tilde A=(a_0,a_1,-a_2,-a_3,a_4).\]
Under this frame
change, the invariants $p_1, p_2, p_3, q_1, q_2$ transform by
$$
\tilde q_1= -q_1,\quad   \tilde q_2= -q_2,\quad
\tilde p_1= p_1, \quad \tilde p_2= p_2,  \quad\tilde p_3= p_3.
$$
Thus, there are well defined global functions $\text{\sc j}$,
$\text{\sc w}: M \to \R$ such that
locally
$$
 \text{\sc j} =\frac{1}{2}(p_1-p_3), \quad
  \text{\sc w}= \frac{1}{2}(p_1+p_3).
   $$

We recall that a nondegenerate Legendrian
immersion $F : M \to  \Lambda$ is {\it $L$-minimal} if and only
if $p_1+p_3 = 0$ on $M$ (cf. \cite{MN-TAMS}).

\end{remark}

\subsection{$L$-isothermic surfaces}\label{ss:iso}

We recall that a nondegenerate Legendrian immersion $F : M \to \Lambda$
is {\it $L$-isothermic}
if there exist local coordinates which simultaneously diagonalize
the definite pair of quadratic forms
$\langle da_0,da_0\rangle =(\alpha^2_0)^2+(\alpha^3_0)^2$ and
$\langle da_0,da_1\rangle =(\alpha^2_0)^2-(\alpha^3_0)^2$
and which are
isothermal with respect to $\langle da_0,da_0\rangle$.
If $f : M \to \R^3$ is an immersed surface
without umbilic and parabolic points, oriented by the unit normal field $n$,
the $L$-isothermic condition amounts to the existence of
isothermal (conformal)
curvature line coordinates for the pair of quadratic forms $III=dn \cdot dn$
and $II=df\cdot dn$.
In \cite{MN-BOLL}, it is shown that $F$ is $L$-isothermic if and only if
$p_2$ vanishes identically on $M$. In this case, there exist isothermal
curvature line coordinates $z= x+iy$
such that the middle coframe
$(\alpha^2_0, \alpha^3_0)$ takes the form
\[
 \alpha^2_0= e^u dx,\quad \alpha^3_0 = e^u dy,
  \]
for a smooth function $u$ on $M$.
The function $\Phi =e^u$ is called the {\it Blaschke potential} of $F$.

Accordingly, from \eqref{se0}, \eqref{se1} and \eqref{se2} it follows that
\begin{eqnarray}
  q_1= -e^{-u}u_y, & {} &q_2= e^{-u}u_x, \label{iso1} \\
   p_1 -p_3 &=& - e^{-2u} \Delta u \label{iso2}.
    \end{eqnarray}
Moreover, using \eqref{se3} and \eqref{se4} yields
\begin{equation}\label{dW}
 \begin{split}
  d\left(e^{2u}(p_1 +p_3)\right) &= -e^{2u}\left\{\left(e^{-2u}\Delta u\right)_x
   +4u_x(e^{-2u}\Delta u)\right\} dx \\
   &\quad +
    e^{2u}\left\{(e^{-2u}\Delta u)_y
     +4u_y(e^{-2u}\Delta u)\right\}dy.
  \end{split}
     \end{equation}
The integrability condition of \eqref{dW} is the so-called {\it Blaschke equation},
\begin{equation}\label{blaschke-eq}
 \Delta \left(e^{-u}(e^u)_{xy}\right) =0,
  \end{equation}
which can be viewed as the completely integrable (soliton)
equation of $L$-isothermic surfaces \cite{MN-BUD, MN-AMB}.

Conversely, let $U$ be a simply connected domain in $\C$, and let $\Phi = e^u$ be
a solution to the Blaschke equation \eqref{blaschke-eq}. It follows
that the right hand side of \eqref{dW} is a closed 1-form, say $\eta_\Phi$.
Thus, $\eta_\Phi = dK$, for some function $K$ determined up to an
additive constant.
If we let
\begin{equation}\label{J-L}
    \text{\sc w} = Ke^{-2u}, \quad
   \text{\sc j} = -\frac{1}{2}e^{-2u}\Delta u,
   \end{equation}
the 1-form defined by
\[
\alpha=
\left(\!\left(\!\begin{smallmatrix}
0\\
e^udx\\
e^udy\\
0\\
\end{smallmatrix}\!\right)\!,\!
\left(\!\begin{smallmatrix}
 2du&(\text{\sc w}+\text{\sc j})e^udx&(\text{\sc w}-\text{\sc j})e^udy&0 \\
 e^udx& 0 &u_ydx - u_xdy &(\text{\sc w}+\text{\sc j})e^udx \\
 -e^udy&-u_ydx +u_xdy&0&(\text{\sc w}-\text{\sc j})e^udy\\
 0&e^udx&-e^udy&-2du
\end{smallmatrix}\!\right)\!\right)\!
\]
satisfies the Maurer--Cartan integrability condition
$$d\alpha +\alpha\wedge\alpha = 0$$
and then integrates to a
map $A=(x,a) : U \to L$, such that $dA = A\alpha$.
The map $F : U \to \Lambda$ defined by
$$F= [x,a_1]$$
is a smooth Legendre immersion and $A$
is a middle frame field along $F$. Thus, $F$ is an $L$-isothermic
immersion (unique up to
Laguerre equivalence) and $\Phi$ is its Blaschke potential.

If, for any $m\in \R$, we let
\begin{equation}\label{J-L-m}
    \text{\sc w}_m = \text{\sc w}+me^{-2u}, \quad
   \text{\sc j}_m = \text{\sc j}  = -\frac{1}{2}e^{-2u}\Delta u,
   \end{equation}
then the 1-form defined by
\[
\alpha^{(m)}=
\left(\!\left(\!\begin{smallmatrix}
0\\
e^udx\\
e^udy\\
0\\
\end{smallmatrix}\!\right)\!,\!
\left(\!\begin{smallmatrix}
 2du&(\text{\sc w}_m+\text{\sc j}_m)e^udx&(\text{\sc w}_m-\text{\sc j}_m)e^udy&0 \\
 e^udx& 0 &u_ydx - u_xdy &(\text{\sc w}_m+\text{\sc j}_m)e^udx \\
 -e^udy&-u_ydx +u_xdy&0&(\text{\sc w}_m-\text{\sc j}_m)e^udy\\
 0&e^udx&-e^udy&-2du
\end{smallmatrix}\!\right)\!\right)\!
\]
satisfies the Maurer--Cartan integrability condition
$$d\alpha^{(m)} +\alpha^{(m)}\wedge\alpha^{(m)} = 0,$$
so that there exists a
smooth map $A^{(m)}=(x^{(m)},a^{(m)}) : U \to L$, such that $dA^{(m)} = A^{(m)}\alpha^{(m)}$.
The map $F_m : U \to \Lambda$, given by $F_m= [x^{(m)},a_1^{(m)}]$,
is a smooth Legendre immersion and $A^{(m)}$
is a middle frame field along $F_m$. Thus, $F_m$ is an $L$-isothermic
immersion (unique up to
Laguerre equivalence) with the same Blaschke potential $\Phi$ as $F =F_0$.
Then there exist a 1-parameter family of non-equivalent $L$-isothermic immersions $F_m$,
all of which have the same Blaschke potential $\Phi$.

Actually, any other nondegenerate $L$-isothermic immersion having $\Phi$
as Blaschke potential is Laguerre equivalent to $F_m$, for some $m\in \R$.

\begin{defn}
Two $L$-isothermic immersions
$F, \tilde F$ which are not Laguerre equivalent are said to
be $T$-{\it transforms} ({\it spectral deformations}) of each other
if they have the same Blaschke potential.
\end{defn}

\begin{remark}
The spectral family $F_m$ constructed above describes all $T$-fransforms
of $F= F_0$. In fact, any nondegenerate $T$-transform of $F$ is Laguerre
equivalent to $F_m$, for some $m\in \R$.
Such a 1-parameter family of $L$-isothermic surfaces amounts to the family of
second order Laguerre deformations of $F$ in the sense of Cartan
(\cite{MN-TAMS, MN-BOLL}).
\end{remark}

\subsection{The geometry of the $L$-Gauss map}\label{ss:L-Gauss-map}

Given the identification of the space of $L$-spheres with
Minkowski 4-space $\R^4_1$,
an immersion $\sigma : M \to \R^4_1$ of a surface $M$ into $\R^4_1$
is called a {\it sphere congruence}.
A Legendre
immersion $F = (f,n)$ is said to envelope the sphere congruence $\sigma$
if, for each $p\in M$, the $L$-sphere represented by $\sigma(p)$
and the oriented plane $\pi(n(p), f(p))$ are in oriented contact
at $f(p)$. If $\sigma$ is a spacelike immersion,
there exist two enveloping surfaces \cite{Blaschke}.

For a nondegenerate Legendre immersion $F : M \to \Lambda$,
the $L$-Gauss map
\[
 \sigma_F : M \to \R^4_1, \, p \mapsto \sigma_F(p) : = a_0(p)
  \]
defines a spacelike immersion
(cf. \cite{MN-TAMS})
which corresponds to
the classical {\it middle sphere congruence} of $F$ (cf. \cite{Blaschke, MN-TAMS}).

The middle frame field $A$ along $F$ is adapted to the $L$-Gauss map
$\sigma_F$. In fact,
if $T\R^4_1$ denotes the tangent bundle of $\R^4_1$, then the bundle induced
by $\sigma_F$ over $M$ splits into the direct sum
\[
 \sigma_F^\ast (T \R^4_1) = T(\sigma_F) \oplus N(\sigma_F),
  \]
where $T(\sigma_F) =\span\{a_2, a_3\}$ is the tangent bundle of $\sigma_F$
and $N(\sigma_F) =\span\{a_1, a_4\}$ its normal bundle.

The first fundamental form of $\sigma_F$, i.e., the metric induced by
$\sigma_F$ on $M$, has the expression
\[
 g_\sigma=\langle d\sigma_F, d\sigma_F\rangle
  = (\alpha^2_0)^2 +(\alpha^3_0)^2,
  \]
and $\alpha^2_0$, $\alpha^3_0$ defines an orthonormal coframe field on $M$.

As $d\sigma_F(TM) = \span\{a_2, a_3\}$,
it follows from \eqref{lframe2} that
\[
  \alpha^1_0 = 0 = \alpha^4_0.
  \]
From the exterior derivative of these equations, we have
\[
\aligned
 0 &= d\alpha^1_0 = -\alpha^1_2 \wedge\alpha^2_0 - \alpha^1_3\wedge\alpha^3_0,\\
 0 &= d\alpha^4_0 = -\alpha^4_2 \wedge\alpha^2_0 - \alpha^4_3\wedge\alpha^3_0
\endaligned
\]
and then, by Cartan's Lemma,
\[
 \alpha^{\nu}_i = h^{\nu}_{i2}\alpha^2_0 + h^{\nu}_{i3}\alpha^3_0,
 \quad  h^{\nu}_{ij}=h^{\nu}_{ji} \quad \nu=1,4;\, i,j =2,3,
  \]
where the functions $h^{\nu}_{ij}$ are the components of the
second fundamental form of $\sigma_F$,
\[
 \Pi = \sum_{i,j=2,3} h^1_{ij}\alpha^i_0\alpha^j_0\otimes a_4
 + \sum_{i,j=2,3} h^4_{ij}\alpha^i_0\alpha^j_0 \otimes a_1.
   \]
From \eqref{conn-form}, it follows that
\[
  (h^1_{ij})= \begin{pmatrix} 1& 0\\0& -1 \end{pmatrix},
  \quad  (h^4_{ij})= \begin{pmatrix} p_1& p_2\\p_2&p_3 \end{pmatrix}.
   \]
The mean curvature vector of $\sigma_F$ is half the
trace of $\Pi$
with respect to $g_\sigma$,
\[
  2 \mathbf{H} = (p_1 + p_3) a_1.
  \]

\begin{remark}
Note that $\mathbf{H}$ is a null section of the
normal bundle $N(\sigma_F)$, i.e., $\langle\mathbf{H}, \mathbf{H} \rangle=0$.
Moreover, $\mathbf{H}\equiv 0$ on $M$  if and only if
$p_1+ p_3$ vanishes identically on $M$ if and only if
the Legendrian immersion $F: M \to \mathbf\Lambda$ is $L$-minimal
(cf. \cite{MN-TAMS}).
\end{remark}


With respect to the null frame field $\{a_1, a_4\}$,
the normal connection $\nabla^\perp$ in the normal bundle $N(\sigma)$
of $\sigma$ is
given by
\[
 \nabla^\perp a_1 =  \alpha^1_1 \otimes a_1, \quad
  \nabla^\perp a_4 =  -\alpha^1_1 \otimes a_4.
  \]
In particular, we have
\begin{equation}\label{H-normal-deriv}
 2\nabla^\perp \mathbf{H} =\left[d(p_1 +p_3) + (p_1 +p_3) \,\alpha^1_1\right] a_1,
   \end{equation}
so that the parallel condition $\nabla^\perp \mathbf{H} = 0$ takes the form
\begin{equation}\label{gLm-eq}
 d(p_1+p_3) + 2(p_1+p_3)(q_2\alpha^2_0 -q_1 \alpha^3_0) = 0.
  \end{equation}


\section{The proof of Theorem \ref{thm:A}}\label{s:thm:A}

In this section, for any nondegenerate Legendre immersion $F : M \to \Lambda$,
we introduce a quartic differential $\Q_F$ and a quadratic differential $\P_F$.
Theorem \ref{thm:A} will be proved using some results
about Legendre immersions with holomorphic $\Q_F$ (cf. Section \ref{ss:hol-diff})
and the interpretation of such immersions as $T$-transforms of $L$-minimal isothermic
surfaces (cf. Section \ref{ss:s-L-iso}).

\subsection{Holomorphic differentials for Legendre immersions}\label{ss:hol-diff}
Let $M$ be an oriented surface and let $F : M \to \Lambda$
be a nondegenerate Legendrian immersion into the Laguerre space.
Let $A : M\to L$ be the middle frame field along $F$ and let $\alpha = A^{-1}dA$
denote its Maurer--Cartan form.
The metric $(\alpha^2_0)^2 + (\alpha^3_0)^2$ and the area element
$\alpha^2_0\wedge\alpha^3_0$ induced by $A$ determine on $M$ an oriented conformal
structure and hence, by the existence of isothermal coordinates, a unique compatible
complex structure which makes $M$ into a Riemann surface. In terms of the middle
frame
field $A$, the complex structure is characterized by the property that
the complex-valued  1-form
\begin{equation}\label{1-0form}
 \varphi = \alpha^2_0 +i \alpha^3_0
  \end{equation}
is of type $(1,0)$.

\begin{defn}
Let $F : M\to \Lambda$ be a nondegenerate Legendrian immersion.
The complex-valued quartic differential form given by
\begin{equation}\label{quartic}
 \Q_F = Q  \varphi^4, \quad Q := \frac{1}{2}(p_1 - p_3) - i p_2,
  \end{equation}
and the complex-valued quadratic differential form given by
\begin{equation}\label{quadratic}
 \P_F = P \varphi^2, \quad P :=  p_1 + p_3
  \end{equation}
are globally defined on the Riemann surface $M$.
\end{defn}

\begin{remark}
The quartic differential $\Q_F$ was considered by the authors
for $L$-minimal surfaces \cite{MN-TAMS}. For $L$-minimal surfaces,
$\Q_F$ is holomorphic. The quadratic differential
$\P_F$ vanishes exactly for $L$-minimal surfaces.
\end{remark}

We now collect some useful facts about these differentials.
We begin with a simple observation.

\begin{lemma} The quartic differential $\Q_F$ is holomorphic if and only if
\begin{equation}\label{cns-holQ}
 dQ \wedge \varphi = -4(q_2\alpha^2_0 - q_1\alpha^3_0) Q \wedge \varphi.
  \end{equation}
\end{lemma}

\begin{proof}
Taking the exterior derivative of \eqref{1-0form}
and using the structure equations give
\begin{equation}\label{d-omega}
 d \varphi = (q_2\alpha^2_0 - q_1\alpha^3_0) \wedge \varphi.
 \end{equation}
Let $z$ be a local complex coordinate on $M$, so that
\begin{equation}\label{omega-z}
\varphi = \lambda dz, \quad \lambda \neq 0.
\end{equation}
Then, locally,
\[
 \Q_F = Q\lambda^4 (dz)^4.
 \]
Exterior differentiation of \eqref{omega-z} and use of \eqref{d-omega} give
\begin{equation}\label{d-omega-z}
 \left(d\lambda - \lambda (q_2\alpha^2_0 - q_1\alpha^3_0) \right) \wedge \varphi = 0.
  \end{equation}
By \eqref{d-omega-z}, it is easily seen that condition \eqref{cns-holQ} holds
if and only if
\[
d(Q\lambda^4)\wedge \varphi =
\lambda^4\left[dQ + 4(q_2\alpha^2_0 - q_1\alpha^3_0)Q\right]\wedge \varphi = 0,
\]
that is, if and only if $\frac{\partial}{\partial \bar z} (Q\lambda^4) = 0$.

\end{proof}

Next, we prove the following.

\begin{prop}\label{prop:hol-cond}
Let $F: M \to \Lambda$ be a nondegenerate Legendrian immersion.
Then:

\begin{enumerate}

\item $\Q_F$ is holomorphic
if and only if the $L$-Gauss map of $F$ has parallel mean curvature
vector.

\item If $\Q_F$ is holomorphic, then
$\P_F$ is holomorphic.

\item If $\Q_F$ is holomorphic and $\P_F$ is non-zero, then $F$
is $L$-isothermic. 

\item If $\Q_F= Q\varphi^4$ is holomorphic and $\P_F = P\varphi^2 \neq 0$, then
$$Q = c P^2,$$
for a real constant $c$.

\end{enumerate}
\end{prop}

\begin{proof} (1) It suffices to prove that \eqref{cns-holQ} is equivalent
to the parallel condition equation \eqref{gLm-eq}.
Writing out the left and right hand side
of \eqref{cns-holQ} using the structure equations, we get
\[
\begin{split}
 &dQ \wedge \varphi = -\frac{1}{2}(dp_1 +dp_3) \wedge \alpha^2_0
+ \frac{i}{2}(dp_1+dp_3) \wedge\alpha^3_0\\
 & \quad + (q_1p_3-3q_1p_1 -4q_2p_2) \alpha^2_0\wedge \alpha^3_0
+i (3q_2p_3 +4q_1p_2 -q_2p_1)\alpha^2_0\wedge \alpha^3_0
\end{split}
\]
and
\[
\begin{split}
-4(q_2\alpha^2_0 - q_1\alpha^3_0) Q \wedge \varphi =&
\left[2q_1p_3-2q_1p_1 -4q_2p_2\right]\alpha^2_0\wedge \alpha^3_0\\
 &\quad +i\left[2q_2p_3+4q_1p_2-2q_2p_1)\right]\alpha^2_0\wedge \alpha^3_0.
\end{split}
 \]
Thus, \eqref{cns-holQ} is equivalent to
\begin{eqnarray}
-\frac{1}{2}(dp_1 +dp_3) \wedge \alpha^2_0 -
(q_1p_1 +q_1p_3) \alpha^2_0\wedge \alpha^3_0 &=& 0,\nonumber \\
\frac{1}{2}(dp_1 +dp_3) \wedge \alpha^3_0
 + (q_2p_3 +q_2p_1) \alpha^2_0\wedge \alpha^3_0 &=& 0,\nonumber
\end{eqnarray}
which in turn is equivalent to the parallel condition
\[
d(p_1+p_3) + 2(p_1+p_3)(q_2\alpha^2_0 -q_1 \alpha^3_0) = 0,
\]
as claimed.

(2) Observe that the exterior derivative of $\varphi$ can be written as
\begin{equation}\label{d-omega-quadratic}
 d \varphi = (q_2\alpha^2_0 - q_1\alpha^3_0) \wedge \varphi.
\end{equation}
By reasoning as above,
 $\P_F$ is holomorphic if and only if
\begin{equation}\label{cns-holP}
 dP \wedge \varphi = -2 (q_2\alpha^2_0 - q_1\alpha^3_0) P \wedge \varphi.
  \end{equation}
The claim follows
from the condition $dP + 2P(q_2\alpha^2_0 - q_1\alpha^3_0)=0$,
which amounts to the condition that $\Q_F$ be holomorphic.

(3) If $\P_F$ is non-zero, i.e., $F$ is not $L$-minimal, it follows
from \eqref{gLm-eq} that $d\alpha^1_1 =0$. On the other hand,
the structure equations give
\[
 d\alpha^1_1  = 2p_2 \alpha^2_0 \wedge \alpha^3_0,
 \]
which implies $p_2 = 0$. Thus $F$ is $L$-isothermic.

(4) Under the given hypotheses, it follows
from assertion (3)
that $p_2 = 0$
and then that condition \eqref{cns-holQ} can be written
\[
 dQ  + 4\mu Q \equiv 0, \mod \varphi,
  \]
where $\mu =q_2\alpha^2_0 - q_1\alpha^3_0$ and $d\varphi = \mu \wedge \varphi$.
Moreover, condition \eqref{cns-holP} that $\P_F$ be holomorphic
can be written
$$
  dP + 2\mu P \equiv 0, \mod \varphi.
   $$
Actually, $dP + 2\mu P=0$.
It then follows that
$$
  d\left(\frac{Q}{P^2}\right) \equiv 0,\mod \varphi.
   $$
This proves that the real-valued function $Q/P^2$ is holomorphic, and hence
a constant function, as claimed.
\end{proof}

We are now ready to prove our next result.

\begin{prop}\label{thm:hQ-iff-W-or-special}
The quartic differential $\Q_F$ of a nondegenerate Legendre immersion
$F : M \to \Lambda$ is holomorphic if and only if
the immersion is $L$-minimal,
in which case the quadratic differential $\P_F$ vanishes on $M$,
or is $L$-isothermic
with Blaschke potential
$\Phi =e^u$
satisfying the second order partial differential equation
\begin{equation}\label{special-eq}
 \Delta u  = ce^{-2u}, 
  \end{equation}
where $c$ is a real constant.
\end{prop}

\begin{proof}
If $\Q_F$ is holomorphic and the holomorphic quartic differential $\P_F$
vanishes, then $F$ is $L$-minimal.
If instead $\P_F$ is nowhere vanishing, then
$F$ is $L$-isothermic by Proposition \ref{prop:hol-cond} (3).
Let $z = x+iy : U\subset M  \to \C$ be an isothermic
chart, so
that the middle coframing
$(\alpha^2_0, \alpha^3_0)$ takes the form
$\alpha^2_0= e^udx$ and $\alpha^3_0 = e^udy$,
where $\Phi =e^u$ is the Blaschke potential  (cf. Section \ref{ss:iso}).

 From \eqref{iso1} and \eqref{iso2} we get
\[
 \Q_F = \frac{1}{2}(p_1 - p_3) \omega^4 =  \text{\sc j}e^{4u} (dz)^4
=-\frac{1}{2}(e^{-2u} \Delta u) e^{4u}(dz)^4.
   \]
Since $\Q_F$ is holomorphic,
\[
(e^{-2u} \Delta u) e^{4u} = c,
  \]
for a constant $c\in \R$, that is
\[
 \Delta u  = ce^{-2u}.
  \]

Conversely, if we assume that $\Phi = e^{u}$ satisfy the equation
\eqref{special-eq}, then
the right hand side of \eqref{dW} vanishes identically, which implies
that $p_1 + p_3 = k e^{-2u}$, for a constant $k\in \R$.
A direct computation shows that $p_1 + p_3 = k e^{-2u}$ satisfies
the equation
\[
 d(p_1+p_3) + 2(p_1+p_3)(q_2\alpha^2_0 -q_1 \alpha^3_0) = 0.
	\]
This expresses the fact that the $L$-Gauss map
of $F$, $\sigma_F = a_0$, has
parallel mean curvature vector, or equivalently,
that the quartic differential $Q_F$ is holomorphic.
\end{proof}

\subsection{Special $L$-isothermic surfaces and Laguerre deformation}\label{ss:s-L-iso}

\begin{defn}
A nondegenerate $L$-isothermic immersion $F : M \to \Lambda$
is called {\it special} if its Blaschke potential $\Phi = e^u$
satisfies the second order partial differential equation \eqref{special-eq}
of Theorem \ref{thm:hQ-iff-W-or-special}, i.e.,
\[
 \Delta u  = ce^{-2u}, \quad c\in \R.
  \]
The constant $c$ is called the {\it character} of the special $L$-isothermic
surface $F$.
\end{defn}


\begin{ex}[$L$-minimal isothermic surfaces]

In terms of the Laguerre invariants, $L$-minimal surfaces are characterized
by the condition $p_1 +p_3 = 0$ (cf. \cite{MN-TAMS}). Therefore, if a
nondegenerate $L$-isothermic immersion $F : M \to \Lambda$ is also $L$-minimal,
the right hand side of \eqref{dW} is identically zero. This implies
\[
 d\left(e^{2u} \Delta u\right)= 0,
  \]
and hence the following.

\begin{prop}\label{prop:L-min-iso-are-special}
 Any nondegenerate $L$-mi\-ni\-mal
isothermic immersion $F : M \to \Lambda$ is special $L$-isothermic.
\end{prop}

Other examples of $L$-minimal isothermic surfaces include
$L$-minimal canal surfaces \cite{MN-REND-RM, MN-AMB}.

\end{ex}

\subsubsection{Special $L$-isothermic surfaces as $T$-transforms}

Let $F : M \to \Lambda$ be a special $L$-isothermic immersion. From
the proof of
Theorem \ref{thm:hQ-iff-W-or-special}, we have that the invariants $\text{\sc j}$ and
$\text{\sc w}$ of $F$ are given by
\begin{equation}\label{special-L-J}
   \text{\sc w} = ke^{-2u}, \quad
   \text{\sc j} = -\frac{1}{2}e^{-2u}\Delta u,
   \end{equation}
where $k$ is a real constant.

\begin{defn}
The constant $k$ will be referred to as the
{\it deformation} (or {\it spectral}) {\it parameter} of the special
$L$-isothermic immersion $F$.
\end{defn}

We have the following.

\begin{prop}\label{prop:special-as-Ttrans}
Any special $L$-isothermic immersion in Laguerre space is the $T$-transform of
an $L$-minimal isothermic immersion.
\end{prop}

\begin{proof}
According to Section \ref{ss:iso}, there exists, up to Laguerre equivalence,
a unique $L$-isothermic immersion $F$ with
Blaschke potential $\Phi = e^u$ satisfying \eqref{special-eq}
and with invariant functions
\[
   \text{\sc w} = 0, \quad
   \text{\sc j} = -\frac{1}{2}e^{-2u}\Delta u = -\frac{c}{2}e^{-4u}.
    \]
Since $\text{\sc w} = 0$, we have that $F$ is $L$-minimal.
Next, let $\tilde F$ be a special $L$-isothermic immersion with the same Blaschke
potential $\Phi$ as $F$
and with deformation parameter $k$. The discussion in Section \ref{ss:iso}
implies that $\tilde F$ is a $T_m$-transform of $F$. The invariants of $\tilde F$
are then given by
\begin{equation}\label{special-L-J-m}
   \text{\sc w}_m = me^{-2u}, \quad
   \text{\sc j}_m = \text{\sc j}= -\frac{c}{2}e^{-4u}.
  \end{equation}
From \eqref{special-L-J} and \eqref{special-L-J-m}, it follows that $m=k$.
\end{proof}

From Proposition \ref{thm:hQ-iff-W-or-special},
Proposition \ref{prop:L-min-iso-are-special}, and
Proposition \ref{prop:special-as-Ttrans}, we get the first main results
of the paper.

\vskip0.3cm
\noindent\textbf{Theorem A.}
\textit{The quartic differential $\Q_F$ of
a nondegenerate Legendre immersion $F : M \to  \Lambda$ is holomorphic
if and only if the immersion $F$ is $L$-minimal, in which case $\P_F$ vanishes,
or is locally the $T$-transform of an $L$-minimal isothermic
surface.}

\vskip0.3cm
In particular, if $F$ has holomorphic $\Q_F$ and zero $\P_F$, then
$F$ is $L$-isothermic if and only if it is $L$-minimal isothermic.

\section{The proof of Theorem \ref{thm:B}}\label{s:thm:B}

In this section we characterize
$L$-minimal isothermic surfaces and
their $T$-transforms (i.e., special $L$-isothermic
surfaces with non-zero deformation parameter)
in terms of the geometry of their $L$-Gauss maps.
Theorem \ref{thm:B} will be proved using
these characterizations,
which are given, respectively, in
Proposition \ref{prop:L-min-iso} and Proposition \ref{p:main}.

\subsection{The geometry of $L$-minimal isothermic surfaces}\label{ss:geo-min-iso}

The property of being $L$-minimal and $L$-isothermic is reflected in the
differential geometry of the $L$-Gauss map of $F$.
%
In the following result, the terminology used for hyperplanes of $\R^4_1$ is that introduced in Section \ref{ss:cyclo}.

\begin{prop}\label{prop:L-min-iso}
A nondegenerate $L$-minimal immersion $F : M \to \Lambda$
is $L$-isothermic if and only if its $L$-Gauss map $\sigma_F : M \to \R^4_1$
is restricted to lie in the hyperplane of $\R^4_1$ defined by the equation
\[
  \langle \sigma_F -O, v \rangle =0,
   \]
   for some point $O$ and some constant vector $v$.
%
%
In particular,
the $L$-Gauss map $\sigma_F$ of a nondegenerate $L$-mi\-ni\-mal
isothermic immersion $F$ has zero mean curvature
in some spacelike, timelike, or (degenerate) isotropic hyperplane of $\R^4_1$.
\end{prop}

\begin{proof}
Let $A=(a_0,a)$ be a middle frame field along $F$ and let $z=x+iy$ be an
isothermic
chart, so that $\alpha^2_0 = e^udx$, $\alpha^3_0 = e^udy$, where
$\Phi =e^{u}$ is the Blaschke potential.
Since $p_1 + p_3 =0$ and $p_2 = 0$, by \eqref{se3} and \eqref{se4}, we have
\begin{equation}\label{dp1}
  dp_1 + 2p_1\alpha^1_1 =0.
  \end{equation}
Next, define
 \[
   v : = e^{2u}\left( -p_1 a_1 +a_4 \right).
    \]
By exterior differentiation of $v$ and use of \eqref{dp1},
it is easily verified that $dv = 0$, i.e., $v$ is a constant vector.
This, combined with the fact that $d\sigma_F = \alpha^2_0a_2 +\alpha^3_0a_3$, gives
\[
 d \langle \sigma_F, v \rangle =\langle d\sigma_F, v \rangle = 0,
\]
that is,
\[
  \langle \sigma_F -O, v \rangle =0,
   \]
for some point $O\in \R^4_1$, which implies that $\sigma_F$ actually lies in the
hyperplane of $\R^4_1$ defined by $O$ and the vector $v$.
Depending on whether $v$ is timelike, spacelike, or isotropic, $\sigma_F$ lies
in a spacelike, timelike, or (degenerate) isotropic hyperplane of $\R^4_1$.

Conversely, if $\langle \sigma_F -O, v \rangle =0$, for some point $O\in \R^4_1$ and
some constant vector $v$, then $\langle d\sigma_F, v \rangle =0$, which implies
$v = \ell_1 a_1 + \ell_4 a_4$, for some smooth functions $\ell_1, \ell_4$.
Exterior differentiation of $v = \text{const}$ and use of the structure equations yields
\begin{eqnarray*}
d\ell_1 +\ell_1 \alpha^1_1 =0,  && (\ell_1  + \ell_4  p_1)\alpha^2_0 + \ell_4 p_2 \alpha^3_0 =0,
 \\
 \ell_4 p_2 \alpha^2_0 - (\ell_1  + \ell_4  p_1)\alpha^3_0 = 0,  &&
     d\ell_4 -\ell_4 \alpha^1_1 = 0,
      \end{eqnarray*}
from which follows that $d\alpha^1_1 =0$. On the other hand,
by \eqref{conn-form} and \eqref{se2},
$$
  d\alpha^1_1 = 2 p_2 \alpha^2_0 \wedge \alpha^3_0,
   $$
which implies $p_2 = 0$, and hence $F$ is $L$-isothermic.

The last claim follows from the fact that
$\sigma_F$ lies in
a hyperplane, a totally geodesic submanifold,
and from the fact that, being $F$ $L$-minimal, $\sigma_F$ has
zero mean curvature vector, that is, $\mathbf H =0$
(cf. also \cite{AGM}, Remark 3).
\end{proof}

\begin{remark}
From the previous proof, it follows that
\[
  \langle v, v \rangle = 2p_1 e^{4u},
  \]
and that the equation \eqref{special-eq} satisfied by the Blaschke potential
$\Phi=e^{u}$ becomes
\begin{equation}\label{delta-u=c}
 \Delta u = - \langle v, v \rangle e^{-2u}.
\end{equation}
Thus, according to whether $p_1$ is negative, positive, or zero,
$\sigma_F$ lies in some spacelike, timelike, or isotropic hyperplane of $\R^4_1$.
\end{remark}

\begin{remark}
In the language of Section \ref{ss:cyclo}, Proposition \ref{prop:L-min-iso}
says that the $L$-spheres represented
by the $L$-Gauss map of an
$L$-minimal isothermic surface are restricted to lie in a
planar system of $L$-spheres.
The description of $L$-minimal isothermic surfaces goes back to
the work of Blaschke (cf. \cite{Blaschke1} (1925) and \cite{Blaschke}, $\S$ 81),
where it is proved that, up to $L$-equivalence, they either correspond
to minimal surfaces in Euclidean space, surfaces whose middle sphere
congruence is tangent to a fixed plane in Euclidean space, or surfaces
whose middle spheres have centers lying on a fixed plane. More recently,
it has been proved (cf. \cite{Song2013}) that $L$-minimal isothermic
surfaces are locally Laguerre equivalent to surfaces with vanishing
mean curvature in $\R^3$, $\R^3_1$, or a (degenerate) isotropic 3-space $\R^3_0$
of signature $(2,0)$.
See also \cite{Szer} for other results on
$L$-minimal isothermic surfaces.
\end{remark}

\subsection{The geometry of special $L$-isothermic surfaces}\label{ss:geo-sL-iso}

We now characterize special $L$-isothermic surfaces with
non-zero deformation parameter in terms
of their $L$-Gauss maps. This is given by the following result.

\begin{prop}\label{p:main}
Let $F : M \to \Lambda$ be a nondegenerate Legendre immersion.
The following two statements are equivalent:
\begin{enumerate}

\item $F$ has holomorphic $\Q_F$ and non-zero $\P_F$.

\item
$F$ is $L$-isothermic and its $L$-Gauss map $\sigma_F$ is restricted
to lie on the hypersurface of $\R^4_1$ defined by the equation
 \begin{equation}\label{sigma-in-sph-sys}
  \langle \sigma_F -O, \sigma_F -O \rangle = \mathrm{costant},
   \end{equation}
for some point $O$ of $\R^4_1$.

\end{enumerate}
\end{prop}

\begin{proof} Let us show that $(2)$ implies $(1)$.
Let $A = (a_0, a)$ be a middle frame field along $F$.
Differentiation of
\eqref{sigma-in-sph-sys} yields
\[
  \langle d\sigma_F , \sigma_F -O \rangle = 0,
\]
which implies
\begin{equation}\label{sigma-ort}
 \sigma_F -O  =\ell_1  a_1 + \ell_4  a_4,
  \end{equation}
for smooth functions $\ell_1$, $\ell_2$.
Differentiation of \eqref{sigma-ort} and use of the structure equations,
taking into account that $p_2 = 0$, yields
\[
\begin{split}
 \alpha^2_0\, a_2 +\alpha^3_0\, a_3 &=  d\ell_1 a_1  + d\ell_4 a_4 +\ell_1  da_1 +
\ell_4 da_4  \\
 &=  \left(d\ell_1 +\ell_1 \alpha^1_1\right)a_1 +
  \left(\ell_1 +\ell_4 p_1\right) \alpha^2_0\,a_2  \\
  & \quad + \left(-\ell_1 +\ell_4 p_3\right) \alpha^3_0\,a_3 +
  \left(d\ell_4 -\ell_4 \alpha^1_1\right)a_4,
   \end{split}
     \]
which amounts to
\begin{eqnarray}
   \ell_1  + \ell_4  p_1 =1, & &
  - \ell_1  + \ell_4  p_3 = 1, \label{l1l2a}\\
   d\ell_1 +\ell_1 \alpha^1_1 =0, & &
     d\ell_4 -\ell_4 \alpha^1_1 = 0. \label{l1l2b}
      \end{eqnarray}
The consistency condition of \eqref{l1l2a}, $p_1 + p_3 \neq 0$, implies
$\P_F$ non-zero. Solving \eqref{l1l2a} for $\ell_1$, $\ell_4$, we get
\[
  \ell_1 = \frac{p_3-p_1}{p_1+p_3}, \quad \ell_4 = \frac{2}{p_1 +p_3}.
  \]
Now, it is readily seen that equation $d\ell_4 -\ell_4 \alpha^1_1=0$
amounts the condition that $\Q_F$ be holomorphic.
Using this and Proposition \ref{prop:hol-cond} (4),
one checks that equation $d\ell_1 +\ell_1 \alpha^1_1 =0$
is identically satisfied.

Conversely,
if (1) holds, by Proposition \ref{prop:hol-cond} (3), $F$ is $L$-isothermic.
 Now, since $\P_F$ is non-zero, equations \eqref{l1l2a} are consistent and
\[
  \ell_1 = \frac{p_3-p_1}{p_1+p_3}, \quad \ell_4 = \frac{2}{p_1 +p_3}.
  \]
By Proposition \ref{prop:hol-cond} (4), we have
$(p_3-p_1) = c(p_1+p_3)^2$,
for a constant c, so that
\[
 d\ell_1 +\ell_1 \alpha^1_1 = c\left[d(p_1+p_3) + (p_1+p_3) \alpha^1_1 \right]= 0.
  \]
We also have
\[
 d\ell_4 -\ell_4 \alpha^1_1 = \frac{-2}{(p_1+p_3)^2}
\left[d(p_1+p_3) + (p_1+p_3) \alpha^1_1\right] = 0,
  \]
which implies that equations \eqref{l1l2b} are identically satisfied.
There exist then functions $\ell_1$, $\ell_4$ such that
\[
 d(\sigma_F -\ell_1 a_1 -\ell_4 a_4) = 0.
\]
Thus,
$$\sigma_F -O  =\ell_1  a_1 + \ell_4  a_4,$$
for some point $O\in \R^4_1$, which
is equivalent to the condition \eqref{sigma-in-sph-sys}, as required.
\end{proof}

\begin{remark}
Observe that, with the notation used above,
\begin{equation}\label{curv-hq}
\langle \sigma_F -O, \sigma_F -O \rangle = \mathrm{costant}
= \frac{4(p_1-p_3)}{(p_1+p_3)^2} = \frac{2\textsc j}{\textsc w^2}.
\end{equation}
\end{remark}
\begin{remark}
In the language of Section \ref{ss:cyclo}, the previous proposition
says that the $L$-spheres represented
by the $L$-Gauss map of a nondegenerate Legendre immersion $F$ with
holomorphic $\Q_F$ and non-zero $\P_F$
are restricted to lie in a
spherical system of $L$-spheres.
\end{remark}

For $r>0$ and some point $O\in \R^4_1$, we let
\[
 \mathbb S^3_1(O,r^2) =\left\{x \in \R^4_1 \, : \,
  \langle x-O, x-O\rangle = {1}/{r^2}\right\}
   \]
denote the {\it timelike pseudo-hypersphere}
{\it centered at} $O$, a translate of {de Sitter 3-space}
$\mathbb S^3_1(r^2)\subset \R^4_1$.
The Lorentzian metric on $\R^4_1$ restricts to a Lorentzian metric on
$\mathbb S^3_1(O,r^2)$ having constant sectional curvature $r^2$.

\vskip0.1cm
In the same way, for $r>0$ and some point $O\in \R^4_1$, we let
\[
 \mathbb H^3_0(O,-r^2) =\left\{x \in \R^4_1 \, : \,
  \langle x-O, x-O\rangle = -{1}/{r^2}\right\}
   \]
denote the {\it spacelike pseudo-hypersphere}
{\it centered at} $O$. The Lorentzian metric on $\R^4_1$ restricts to a Riemannian metric
on $\mathbb H^3_0(O,-r^2)$ of constant sectional curvature $-r^2$.
The hyperquadric
$\mathbb H^3_0(O,-r^2)$ consists of two components
congruent to each other under an
isometry of $\mathbb R^4_1$:
the component
$\mathbb H^3_+(O,-r^2)$
through $O + {^t\!\left(\frac{1}{\sqrt2 r},0,0, \frac{1}{\sqrt2 r}\right)}$,
and the component $\mathbb H^3_-(O,-r^2)$
through $O + {^t\!\left(\frac{-1}{\sqrt2 r},0,0, \frac{-1}{\sqrt2 r}\right)}$.
The components $\mathbb H^3_+(O,-r^2)$ and $\mathbb H^3_-(O,-r^2)$
are translates, respectively, of the {\it future} and {\it past embeddings}
of {hyperbolic 3-space} $\mathbb H^3(-r^2)$ in $\R^4_1$.
%
%
\vskip0.1cm
The {\it isotropic pseudo-hypersphere}
{\it centered at} $O$
is the affine lightcone at $O$ defined by
\[
 \mathcal L^3(O) =\left\{x \in \R^4_1 \, : \,
  \langle x-O, x-O\rangle = 0\right\}.
   \]
Equivalently, $\mathcal L^3(O)= O + \mathcal L^3$, where $\mathcal L^3=\left\{x \in \R^4_1 \, : \,
  \langle x, x\rangle = 0\right\}$.
The Lorentzian metric on $\R^4_1$ restricts to a degenerate metric
of signature $(2,0)$ on $\mathcal L^3(O)\setminus \{O\}$.
The two components $\mathcal L^3_+(O)$
and $\mathcal L^3_-(O)$ of the hypersurface
$\mathcal L^3(O)\setminus \{O\}$ are translates of the
time-oriended
lightcones $\mathcal L^3_+$ and $\mathcal L^3_-$ of  $\R^4_1$ (cf. \eqref{time-cone}).

\vskip0.3cm
We are now in a position to prove our second main result.

\vskip0.3cm

\noindent\textbf{Theorem B.}
{\em Let $F : M \to  \Lambda$ be a nondegenerate Legendre immersion.
Then:

\begin{enumerate}

\item
$F$ is $L$-minimal and $L$-isothermic
if and only if
its $L$-Gauss map $\sigma_F : M \to \R^4_1$
has zero mean curvature
in some
spacelike, timelike, or (degenerate) isotropic hyperplane of $\R^4_1$.
%

\item $F$
has holomorphic $\Q_F$ and non-zero $\P_F$
if and only if
its $L$-Gauss map $\sigma_F : M \to \R^4_1$
has constant mean curvature $H=r$
in some $\mathbb H^3_\pm(O,-r^2)\subset \R^4_1$,
$\mathbb S^3_1(O,r^2)\subset \R^4_1$,
or has zero mean curvature in some $\mathcal L^3_\pm(O)\subset \R^4_1$.

\end{enumerate}

\noindent In addition, if the $L$-Gauss map of $F$
takes values in a spacelike (respectively, timelike, isotropic) hyperplane, then the
$L$-Gauss maps of the $T$-transforms of $F$ take values
in a translate of a hyperbolic 3-space (respectively, de Sitter 3-space,
time-oriented lightcone).
}

\begin{proof}
(1) is a consequence of Proposition \ref{prop:L-min-iso}.
(2) From the preceding discussion and by Proposition \ref{p:main},
the $L$-Gauss map $\sigma_F$ takes values
in a component of some $\mathbb H^3_0(O,-r^2)$, in some $\mathbb S^3_1(O,r^2)$, or in a component
of some $\mathcal L^3(O)\setminus \{O\}$, that is,
$\sigma_F$
lies in some (translate of) $\mathbb H^3(-r^2)$,
$\mathbb S^3_1(r^2)$, or in some translate of a time-oriented lightcone of $\R^4_1$.
From this and the fact that $\sigma_F$ has isotropic mean curvature vector field,
i.e., $\langle\mathbf H, \mathbf H \rangle =0$, it follows that
the mean curvature of $\sigma_F$
is constant, of values $H =r$ or zero, as indicated
(cf. also \cite{AGM}, Remark 3).

The last claim is a consequence of \eqref{special-L-J-m}, \eqref{delta-u=c} and \eqref{curv-hq}.
\end{proof}

\section{Laguerre deformation and Lawson correspondence}\label{s:lawson}

In \cite{Lawson}, Lawson proved that there is an isometric correspondence between
certain constant mean curvature surfaces in space forms.
Let $\mathcal M^3(\kappa)$ denote the simply-connected, 3-dimensional
space form of constant curvature $\kappa$. Let $M$ be a simply-connected surface and
let $f_1 : M \to \mathcal M^3(\kappa_1)$ be an immersion of constant mean
curvature $H_1$,
with induced metric $I$ and shape operator $S_1$. Then, for each constant
$\kappa_2 \leq H_1^2 +\kappa_1$, the pair $I$, $S_2 := S_1 + (H_2-H_1) {Id}$
satisfies the Gauss and Codazzi equations for an immersion
$f_2 : M \to \mathcal M^3(\kappa_2)$ of constant mean curvature
$H_2 = \sqrt{H_1^2 +\kappa_1 - \kappa_2}$, which is
isometric to $f_1$.\footnote{Actually,
there exists a $2\pi$-periodic family of isometric immersions $f_{2,\theta}
: M \to \mathcal M^3(\kappa_2)$, the classical associated family.}
The isometric immersions $f_1$, $f_2$ are said to be related by the {\it Lawson
correspondence}. When $f_1$ is a minimal immersion, $f_2$
is also referred to as a constant mean curvature {\it cousin} of $f_1$. In particular,
minimal surfaces in $\mathbb R^3$ (respectively, $S^3$) correspond to constant
mean curvature one surfaces in $\mathbb H^3(-1)$ (respectively, $\mathbb R^3$).
For $\kappa_1 =\kappa_2$ and $H_1=H_2$ we get the family of
associated constant mean curvature $H_1$ surfaces. See \cite{Br1987, UY-Crelle} for
special cases of the Lawson correspondence.

In \cite{Pa1990}, Palmer proved that there exists a Lawson correspondence
between certain constant mean curvature spacelike surfaces in Lorentzian
space forms. In particular, there is a correspondence between
maximal ($H=0$) spacelike surfaces in Minkowski 3-space $\R^3_1$
and spacelike surfaces of constant mean curvature $\pm 1$ in de Sitter
3-space $\mathbb S^3_1(1)$. See \cite{AiAk, AGM, Lee2005} for the
discussion of special cases of this correspondence.

\begin{ex}[Deformation of special $L$-isotermic surfaces with $c > 0$]

Let $F : M \to \Lambda$ be a nondegenerate special $L$-isothermic
surface with Blaschke potential $\Phi = e^u$ satisfying the equation
\[
 \Delta u = c e^{-2u},
  \]
with character $c>0$, and deformation (spectral) parameter $k>0$. This implies
\[
   \text{\sc w} = k e^{-2u}, \quad \text{\sc j} = -\frac{1}{2}e^{-2u}\Delta u.
  \]
According to Proposition \ref{p:main} and \eqref{curv-hq}, the $L$-Gauss map $\sigma$ of $F$
has constant mean curvature $H = \frac{k}{\sqrt c}$ into
(a translate of) the hyperbolic 3-space of constant curvature
$\kappa = -\frac{k^2}{c}$.
For each $m\in \R_+$, consider the $T_m$-transform
$F_m$ of $F$. Again by Proposition \ref{p:main} and \eqref{curv-hq},
the $L$-Gauss map $\sigma_m$ of $F_m$ is restricted
to lie on the hyperquadric centered at $O$ given by
\[
  \langle \sigma_m - O, \sigma_m -O  \rangle =
\frac{2\textsc j_m}{\textsc w_m^2} = -\frac{c}{(m+k)^2},
   \]
for some $O\in \R^4_1$.
Thus, $\sigma_m$ has constant mean curvature $H_m = \frac{m+k}{\sqrt c}$
in (a translate of) the hyperbolic 3-space of curvature
$\kappa_m =-\frac{(m+k)^2}{c}$.
Note that
\[
 \kappa_m +H^2_m = \kappa + H^2 =0
  \]
does not depend on $m$.
We have then established that the $L$-Gauss maps of
the $T$-transforms of special $L$-isothermic surfaces with positive
character and positive deformation parameter
%
all have constant mean curvature in
(a translate of) some hyperbolic 3-space.
Moreover, since the metrics induced by $\sigma_m$ do not depend on
$m$, i.e., $g_{\sigma_m}=g_\sigma$, we may conclude that
the $T$-transformation of such special $L$-isothermic surfaces
can be viewed, via their $L$-Gauss maps,
as the Lawson correspondence between certain
constant mean curvature surfaces in different hyperbolic 3-spaces.

If $F = F_0$ is $L$-minimal isothermic, its $L$-Gauss map
$\sigma_0$ is minimal in (a translate of) Euclidean space $\R^3$.
In this case, for each
$m\in \R^\ast$, the
$L$-Gauss map
$\sigma_m$ has constant mean curvature $m/\sqrt c$
in hyperbolic 3-space $\mathbb H^3(-m^2/c)$.
The family $\{\sigma_m\}_{m\in \R^\ast}$
can be viewed as
the 1-parameter family of isometric immersions associated with the minimal
immersion $\sigma_0$
considered by Umehara--Yamada \cite{UY-Crelle}.
This provides a Laguerre geometric interpretation
of the Umehara--Yamada isometric perturbation
of minimal surfaces in Euclidean space into
constant mean curvature surfaces in hyperbolic 3-space.
A M\"obius geometric interpretation
 of the Umehara--Yamada isometric perturbation
was given in \cite{HMN}.

\end{ex}

\begin{ex}[Deformation of special $L$-isotermic surfaces with $c < 0$]

If $F : M \to \Lambda$ is a nondegenerate special $L$-isothermic
surface with Blaschke potential $\Phi = e^u$ satisfying the equation
\[
 \Delta u = c e^{-2u},
  \]
with character $c<0$, and deformation (spectral) parameter $k>0$, then
\[
   \textsc w = k e^{-2u}, \quad \textsc j = -\frac{1}{2}e^{-2u}\Delta u.
  \]
By Proposition \ref{p:main} and \eqref{curv-hq}, the $L$-Gauss map $\sigma$
has constant mean curvature $H = \frac{k}{\sqrt{-c}}$ into
(a translate of) the de Sitter 3-space of constant curvature
$\kappa = -\frac{k^2}{c}$.
For each $m\in \R_+$,
the $L$-Gauss map $\sigma_m$ of the $T_m$-transform $F_m$ is restricted
to lie on the hyperquadric centered at $O$ given by
\[
  \langle \sigma_m - O, \sigma_m -O  \rangle =
\frac{2\textsc j_m}{\textsc w_m^2} = -\frac{c}{(m+k)^2},
   \]
for some $O\in \R^4_1$.
Thus, $\sigma_m$ has constant mean curvature $H_m = \frac{m+k}{\sqrt {-c}}$
in (a translate of) the de Sitter 3-space of curvature
$\kappa_m =-\frac{(m+k)^2}{c}$.
Note that
\[
 \kappa_m +H^2_m = \kappa + H^2
  \]
does not depend on $m$.
%
We have then established that the $L$-Gauss maps of
the $T$-transforms of special $L$-isothermic surfaces with
negative character and positive deformation parameter
all have constant mean curvature in
(a translate of) some de Sitter 3-space.
As above, the metrics induced by $\sigma_m$ do not depend on
$m$, i.e., $g_{\sigma_m}=g_\sigma$. Thus,
the $T$-transformation of special $L$-isothermic surfaces
with negative character and positive deformation parameter
can be viewed, via their $L$-Gauss maps, as the Lawson correspondence between certain
constant mean curvature spacelike surfaces in different de Sitter 3-spaces.

If $F = F_0$ is $L$-minimal isothermic,
$\sigma_0$ is maximal ($H_0=0$) in (a translate of) Minkowski 3-space $\R^3_1$.
In this case, for each
$m\in \R^\ast$, the $L$-Gauss map $\sigma_m$ has constant mean curvature
$m/\sqrt {-c}$
in de Sitter 3-space $\mathbb S^3(-m^2/c)$.
This provides a Laguerre geometric interpretation
of the Lawson correspondence between
 maximal spacelike surfaces in Minkowski 3-space and
constant mean curvature spacelike surfaces in de Sitter 3-space
(cf. Remark \ref{r:AGM} below).

\end{ex}

\begin{ex}[Deformation of special $L$-isotermic surfaces with $c=0$]

Similar considerations hold for special $L$-isothermic surfaces with
character $c=0$. By considering the $L$-Gauss maps, the $T$-transforms of a
zero mean curvature
spacelike surface in (a translate of) a time-oriented lightcone $\mathcal L^3_\pm \subset \R^4_1$
all have zero mean curvature in (a translate of) $\mathcal L^3_\pm$.
In particular,
if $\sigma_0$ is a zero mean curvature
spacelike surface in a (degenerate) isotropic hyperplane, then,
for each $m\in \R^*$, the $L$-Gauss map $\sigma_m$
has zero mean curvature
in some translate of $\mathcal L^3_\pm$.
As a by-product, the $T$-transformation establishes
an isometric correspondence between zero mean curvature
spacelike surfaces in a (degenerate) isotropic 3-space and zero mean curvature spacelike surfaces
in a time-oriented lightcone of $\R^4_1$.
For a brief introduction to
isotropic geometry we refer to \cite{Pott2009, Pott2012}.

\end{ex}

\begin{remark}\label{r:AGM}

The $L$-Gauss maps of special $L$-isothermic surfaces are examples
of the so-called {surfaces of Bryant type} in $\R^4_1$ (cf. \cite{AGM}): a
spacelike immersion $\psi : M \to \R^4_1$ with isotropic mean
curvature vector $\mathbf H$, i.e.,
$\langle \mathbf H, \mathbf H \rangle =0$,\footnote{otherway said,
$\psi$ is a {\it marginally trapped surface} in $\R^4_1$.}
and flat normal bundle is called a {\it surface of Bryant type} in $\R^4_1$
if $M$ is locally isometric to some minimal surface in $\R^3$
or to some maximal surface in $\R^3_1$.
In the context of surfaces of Bryant type, \cite{AGM} describes
an isometric perturbation of constant mean curvature $H=r$ surfaces
in $\mathbb H^3(-r^2)$ (respectively, $\mathbb S^3_1(r^2)$) to minimal (respectively, maximal)
surfaces in $\R^3$ (respectively, $\R^3_1$), which generalizes that of Umehara-Yamada
\cite{UY-Crelle}. By the above discussion, these isometric deformations can
be viewed as special cases of the Laguerre deformation of $L$-isothermic
surfaces.

\end{remark}

\bibliographystyle{amsalpha}

\begin{thebibliography}{AA}

\bibitem{AiAk}
R. Aiyama, K. Akutagawa, Kenmotsu-Bryant type representation formulas
for constant mean  curvature surfaces
in $H^3(-c^2)$ and $S^3_1(c^2)$,
{\em Ann. Global Anal. Geom.} \textbf{17} (1999), no. 1, 49--75.


\bibitem{AGM}
J. A. Aledo, J. A. G\'alvez, P. Mira,
Marginally trapped surfaces in $\mathbb L^4$ and an extended
Weierstrass-Bryant representation,
{\em Ann. Global Anal. Geom.} \textbf{28} (2005), no. 4, 395--415.


\bibitem{Bianchi1905-12}
L. Bianchi,
Complementi alle ricerche sulle superficie isoterme,
{\em Ann. Mat. Pura Appl.}
\textbf{12} (1905), 19--54.


\bibitem{Blaschke1}
W. Blaschke,
\"Uber die Geometrie von Laguerre: I,
{\em Abh. Math. Sem. Univ. Hamburg} \textbf{3} (1924),
176--194;
II,
{\em Abh. Math. Sem. Univ. Hamburg} \textbf{3} (1924),
195--212;
III,
{\em Abh. Math. Sem. Univ. Hamburg} \textbf{4} (1925)
1--12.

\bibitem{Blaschke}
 W. Blaschke,
{\em Vorlesungen \"uber Differentialgeometrie und geometrische
Grundlagen von Einsteins Relativit\"atstheorie}, B. 3, bearbeitet
von G. Thomsen, J. Springer, Berlin, 1929.

\bibitem{Bobenko2006}
A. I. Bobenko, T. Hoffmann , B. A. Springborn,
Minimal surfaces from circle patterns: geometry from  combinatorics,
{\em Ann. of Math. (2)} \textbf{164} (2006), no. 1, 231--264.

\bibitem{Bobenko2007}
A. I. Bobenko, Y. Suris,
On discretization principles for differential geometry. The geometry of  spheres,
{\em Russian Math. Surveys}  \textbf{62} (2007), no. 1, 1--43.

\bibitem{Bobenko2010}
A. I. Bobenko, H. Pottmann, J. Wallner,
A curvature theory for discrete surfaces based on mesh parallelity,
{\em Math. Ann.} \textbf{348} (2010), no. 1, 1--24.


\bibitem{Bo-Pe2009}
C. Bohle, G. P. Peters,
{Bryant surfaces with smooth ends},
\textit{Comm. Anal. Geom.} \textbf{17} (2009), no. 4, 587--619;
arXiv:math/0411480v2.

\bibitem{Bo2012}
C. Bohle,
{Constant mean curvature tori as stationary solutions
to the Davey--Stewartson equation},
\textit{Math. Z.} \textbf{217} (2012), 489--498.


\bibitem{Br-duality}
R. L. Bryant, A duality theorem for Willmore surfaces, {\em J. Differential Geom.}
\textbf{20} (1984), 23--53.

\bibitem{Br1987}
R. L. Bryant,
Surfaces of mean curvature one in hyperbolic space,
Th\'eorie des vari\'et\'es minimales et applications (Palaiseau,
1983--1984), {\em Ast\'erisque} \textbf{154-155} (1987), 321--347.


\bibitem{Ca1903}
P. Calapso,
Sulle superficie a linee di curvatura isoterme,
Rendiconti Circolo Matematico di Palermo
{\em Rendiconti Circolo Matematico di Palermo} {\bf 17} (1903), 275--286.

\bibitem{Ca1915}
P. Calapso, Sulle trasformazioni delle superficie isoterme,
{\em Ann. Mat. Pura Appl.} {\bf 24} (1915), 11--48.

\bibitem{CM}
D. Carf\`{\i}, E. Musso, T-transformations of Willmore isothermic surfaces,
{\em Rend. Sem. Mat. Messina Ser. II}, suppl. (2000), 69--86.

\bibitem{Ca1}
E. Cartan, {\em Sur le probl\`eme g\'en\'eral de la
d\'eformation}, C. R. Congr\'es Strasbourg (1920), 397--406; or
{\em Oeuvres Compl\`{e}tes}, III 1, 539--548.


\bibitem{Ce}
T. E. Cecil, {\em Lie sphere geometry{:} with applications to
submanifolds}, Springer-Verlag, New York, 1992.


\bibitem{Gr}
P. A. Griffiths, On Cartan's method of Lie groups and moving frames as applied
to uniqueness and existence questions in differential geometry, {\em Duke
Math. J.} {\bf 41} (1974), 775--814.

\bibitem{J}
G. R. Jensen, Deformation of submanifolds of homogeneous spaces,
{\em J. Differential Geom.} {\bf 16} (1981), 213--246.


\bibitem{HMN}
U. Hertrich-Jeromin, E. Musso, L. Nicolodi,
M\"obius geometry of surfaces of constant mean curvature 1 in
hyperbolic space, {\em Ann. Global Anal. Geom.} {\bf 19} (2001),
185--205.

\bibitem{HJlibro}
U. Hertrich-Jeromin, {\em Introduction to M\"obius differential geometry},
London Mathematical Society Lecture Note Series, 300,
Cambridge University Press, Cambridge, 2003.

\bibitem{Kob}
O. Kobayashi, Maximal surfaces in the
3-dimensional Minkowski space $L^{3}$,
{\em Tokyo J. Math.} \textbf{6} (1983), no. 2, 297--309.

\bibitem{Lawson}
H. B. Lawson, Complete minimal surfaces
in $S^{3}$, {\em Ann. of Math. (2)} \textbf{92} (1970), 335--374.


\bibitem{Lee2005}
S. Lee, Spacelike surfaces of constant mean curvature $\pm1$
in de Sitter  3-space ${\mathbb S}^3_1(1)$,
{\em Illinois J. Math.} \textbf{49} (2005), no. 1, 63--98.

\bibitem{li-wang-mm}
T. Li, C. P. Wang, Laguerre geometry of hypersurfaces
in $\mathbb{R}^{n}$, {\em Manuscripta math.} \textbf{122} (2007), 73--95.

\bibitem{MN-REND-RM}
E. Musso, L. Nicolodi,
$L$-minimal canal surfaces,
{\em Rend. Matematica} \textbf{15} (1995), 421--445.


\bibitem{MN-TAMS}
E. Musso, L. Nicolodi,
A variational problem for surfaces in Laguerre geometry,
{\em Trans. Amer. Math. Soc.} \textbf{348} (1996), 4321--4337.

\bibitem{MN-BOLL}
E. Musso, L. Nicolodi,
Isothermal surfaces in Laguerre geometry,
{\em Boll. Un. Mat. Ital. (7) II-B}, Suppl. fasc. 2, {1997},
125--144.

\bibitem{MN-BUD}
E. Musso, L. Nicolodi,
{\em On the equation defining isothermic surfaces in Laguerre geometry},
New Developments in Differential Geometry, Budapest 1996,
Kluver Academic Publishers, Dordrecht, The Netherlands,
285--294.

\bibitem{MN-AMB}
E. Musso, L. Nicolodi, Laguerre geometry of surfaces with plane
lines of curvature, {\em Abh. Math. Sem. Univ. Hamburg} \textbf{69} (1999),
123--138.

\bibitem{MN-IJM}
E. Musso, L. Nicolodi, The Bianchi-Darboux transform of
$L$-isothermic surfaces, {\em Internat. J. Math.} \textbf{11} (2000),
no. 7, 911--924.

\bibitem{MN-TOH}
E. Musso, L. Nicolodi,
Deformation and applicability of surfaces in Lie sphere geometry,
{\em Tohoku Math. J. (2)} \textbf{58} (2006), no. 2, 161--187.

\bibitem{MN-HOUSTON}
E. Musso, L. Nicolodi,
Conformal deformation of spacelike surfaces in Minkowski space,
{\em Houston J. Math.} \textbf{35} (2009), no. 4, 1029--1049.

\bibitem{Pa1999}
B. Palmer,
Remarks on a variational problem in Laguerre geometry,
{\em Rend. Mat. Appl. (7)} \textbf{19} (1999), no. 2, 281--293.


\bibitem{Pa1990}
B. Palmer, Spacelike constant mean curvature surfaces in
pseudo-Riemannian space forms,
{\em Ann. Global Anal. Geom.} \textbf{8} (1990), 217--226.

\bibitem{Pott1998}
H. Pottmann, M. Peternell, Applications of Laguerre geometry in CAGD,
{\em Comput. Aided Geom. Design} \textbf{15} (1998), no. 2, 165--186.

\bibitem{Pott2009}
H. Pottmann, P. Grohs , N. J. Mitra,
Laguerre minimal surfaces, isotropic geometry and linear elasticity,
{\em Adv. Comput. Math.} \textbf{31} (2009), no. 4, 391--419.

\bibitem{Szer}
A. Szereszewski,
$L$-isothermic and $L$-minimal surfaces,
{\em J. Phys. A} \textbf{42} (2009), no. 11, 115203--115217.

\bibitem{Pott2012}
 M. Skopenkov, H. Pottmann, P. Grohs, Ruled Laguerre minimal surfaces,
{\em Math. Z.} \textbf{272} (2012), no. 1-2, 645--674.

\bibitem{Song2013}
Y.-P. Song,
Laguerre isothermic surfaces in ${\mathbb R}^3$ and their Darboux transformation,
{\em Sci. China Math.} \textbf{56} (2013), no. 1, 67--78.

\bibitem{Thomsen}
G. Thomsen, \"Uber konforme Geometrie I: Grund\-lagen der konformen Fl\"a\-chen\-theorie,
{\em Hamb. Math. Abh.} \textbf{3} (1923), 31--56.

\bibitem{UY-Crelle}
M. Umehara, K. Yamada, A parametrization of the Weierstrass formulae
and perturbation of complete minimal surfaces in $\R^3$
into the hyperbolic  3-space, {\em J. Reine Angew. Math.}
\textbf{432} (1992), 93--116.


\end{thebibliography}

\end{document}